\documentclass[11pt]{amsart}
\usepackage{amsfonts}

\usepackage{amsmath,amssymb,amsthm}\allowdisplaybreaks[4]
\usepackage{url}
\usepackage{graphicx}   
\usepackage{verbatim}       
\usepackage{layout}
\usepackage{galois}
\usepackage{mathrsfs}

\newtheorem{theorem}{Theorem}[section]
\newtheorem{lemma}[theorem]{Lemma}
\newtheorem{proposition}[theorem]{Proposition}
\newtheorem{remark}[theorem]{Remark}

\newtheorem{corollary}[theorem]{Corollary}

\newfam\msbmfam\font\tenmsbm=msbm10\textfont
\msbmfam=\tenmsbm\font\sevenmsbm=msbm7
\scriptfont\msbmfam=\sevenmsbm

\def\EE{\mathbb{E}}

\def\RR{\mathbb{R}}

\def\bR{\mathbb{R}}

\def\cP{\mathcal{P}}

\def\al{{\alpha}}\def\be{{\beta}}
\def\ep{{\epsilon}}\def\ga{{\gamma}}
\def\la{{\lambda}}

\def\La{{\Lambda}}

\def\<{\left<}\def\>{\right>}
\def\({\left(}\def\){\right)}

\def\goto{{\rightarrow}}

\def\blemma{\begin{lemma}}\def\elemma{\end{lemma}}
 \def\bproposition{\begin{prop}}\def\eproposition{\end{prop}}
 \def\btheorem{\begin{theorem}}\def\etheorem{\end{theorem}}
 \def\bcorollary{\begin{corollary}}\def\ecorollary{\end{corollary}}

\def\beqlb{\begin{eqnarray}}\def\eeqlb{\end{eqnarray}}
 \def\beqnn{\begin{eqnarray*}}\def\eeqnn{\end{eqnarray*}}

\graphicspath{%
    {converted_graphics/}
    {/}
}
\begin{document}
\title[The compact support property for the $\Lambda$-Fleming-Viot process]
{The compact support property for the $\Lambda$-Fleming-Viot process with
 underlying Brownian motion}
\thanks{The research is supported by NSERC}
\author{Huili Liu and Xiaowen Zhou}
\address{Huili Liu: Department of Mathematics and Statistics, Concordia University,
1455 de Maisonneuve Blvd. West, Montreal, Quebec, H3G 1M8, Canada}
\email{l$\_$huili@live.concordia.ca}
\address{Xiaowen Zhou: Department of Mathematics and Statistics, Concordia University,
1455 de Maisonneuve Blvd. West, Montreal, Quebec, H3G 1M8, Canada}
\email{xzhou@mathstat.concordia.ca } 
\date{\today}
\begin{abstract}
Using the lookdown construction of Donnelly and Kurtz we prove  that, at any fixed
positive time, the  $\La$-Fleming-Viot process with underlying Brownian motion
has a compact support provided that the corresponding $\La$-coalescent
comes down from infinity not too slowly. We also find both upper
bound and lower bound on the Hausdorff dimension for the support.
\end{abstract}
\maketitle \pagestyle{myheadings} \markboth{\textsc{Compact support
property of $\La$-Fleming-Viot process}} {\textsc{H. Liu and X.
Zhou}}
\section{Introduction}
The Fleming-Viot process is a probability-measure-valued stochastic
process for population genetics. It describes the evolution of
relative frequencies for different types of alleles in a large
population that undergoes resampling with possible
mutation. An earlier review on the classical
Fleming-Viot process can be found in Ethier and Kurtz \cite{EtKu93}.
When such a process is treated as a general measure-valued
stochastic process, the mutation can also be interpreted as motion.
In this paper we will consider the Fleming-Viot process with
underlying Brownian motion.

The support property for measure-valued stochastic processes has
been an interesting topic. The earliest work on the compact support
property for classical Fleming-Viot process is due to Dawson and
Hochberg \cite{Dawson}. It was shown in \cite{Dawson} that at any
fixed time $T>0$ the classical Fleming-Viot process with underlying
Brownian motion has a compact support and the support has a
Hausdorff dimension not greater than two. Using non-standard
techniques Reimers \cite{Rei} refined the above result by  proving
that the Hausdorff dimension of support for this process is at most
two simultaneously for all time $t$ from finite interval $[0,T]$.
Iscoe \cite{Iscoe88} first proved the compact support property for
the related Dawson-Watanabe superprocess. The compact support
property for solutions to SPDEs can be found in Mueller and Perkins
\cite{MuPer92}.

The moments for classical Fleming-Viot process with underlying
Brownian motion are determined by a dual process involving Kingman's
coalescent of binary collisions and heat flow. The
$\La$-Fleming-Viot process generalizes the classical Fleming-Viot
process by replacing in its dual process the Kingman's coalescent
with the $\Lambda$-coalescent that allows multiple collisions. Blath
\cite{Bla09} showed that the $\La$-Fleming-Viot process with
underlying Brownian motion does not have a compact support if the
$\La$-coalescent does not come down from infinity. It is then
interesting to know whether such a process allows a compact support
if the $\Lambda$-coalescent comes down from infinity.

In this paper we find a sufficient condition on
$\Lambda$-coalescence rates for the $\La$-Fleming-Viot process to
have a compact support at any fixed positive time. We adapt the idea
of Dawson and Hochberg \cite{Dawson} as follows. Given any fixed
time $T>0$, we can represent the $\La$-Fleming-Viot process at time
$T$ as limit of empirical measures of the exchangeable particle
systems obtained via the  lookdown construction of Birkner and Blath
\cite{BB}. For a sequence of random times ${T}_n$ converging
increasingly to $T$, by the lookdown construction and the property
of coming down from infinity there exist a finite number of common
ancestors at each time ${T}_n$ for those particles at time $T$. Our
assumption on coalescence rates allows us to estimate the number of
common ancestors at time $T_n$. Then locations of the ancestors at
time ${T}_{n+1}$ are determined by a collection of possibly
dependent Brownian motions starting from the locations of ancestors
at time ${T}_n$ and stopping after time $T_{n+1}-T_n$. By the
modulus of continuity for Brownian motion we can estimate the
maximal dislocation of the ancestors at time ${T}_{n+1}$ from those
at time ${T}_n$. Choosing $({T}_n)$ properly and applying
Borel-Cantelli lemma we can show that for $m$ large enough the
maximal dislocations between ${T}_n$ and ${T}_{n+1}$ for all $n\geq
m$ are summable. Then all the particles at time $T$ are situated in
the union of finitely many closed balls centered at the ancestors'
positions at time ${T}_m$ respectively. The compact support property
then follows. As a byproduct of the estimates we can also find an
upper bound for the Hausdorff dimension of the support at time $T$.
The lookdown construction plays a crucial role in our arguments.

The paper is arranged as follows. In Section \ref{s2} we briefly
introduce the $\La$-coalescent and its coming down from infinity
property. In Section \ref{s3} we present the $\La$-Fleming-Viot
process and its lookdown construction. In Section \ref{s5} we prove
the compact support property for $\La$-Fleming-Viot process with underlying
Brownian motion. We also find both upper and lower bounds on the
Hausdorff dimension for the compact support.
\section{The $\La$-coalescent}\label{s2}

\subsection{The $\La$-coalescent}
We first introduce some notations. Put $[n]:=\{1,\ldots,n\}$ and
$[\infty]:=\{1,2,\ldots\}$. An ordered {\it partition} of $D\subset [\infty]$
is a countable collection $\pi=\{\pi_i, i=1,2,\ldots\}$ of disjoint
{\it blocks} such that $\cup_{i}\pi_i=D$ and $\min\pi_i<\min\pi_j$
for $i<j$. Then blocks in $\pi$ are ordered by their least elements.

Denote by $\mathcal{P}_n$ the set of ordered partitions of $[n]$ and by
$\mathcal{P}_\infty$  the set of ordered partitions of $[\infty]$. Write
$\mathbf{0}_{[n]}:=\{\{1\},\ldots,\{n\}\}$ for the partition of
$[n]$ consisting of singletons and $\mathbf{0}_{[\infty]}$ for the
partition of $[\infty]$ consisting of singletons. Given $n\in[\infty]$
and $\pi\in\cP_\infty$, let $R_n(\pi)\in\cP_n$ be the restriction of
$\pi$ to $[n]$.

Kingman's coalescent is a $\mathcal{P}_{\infty}$-valued time
homogeneous Markov process such that all different pairs of blocks
independently merge at the same rate. Pitman \cite{Pitman99} and
Sagitov \cite{Sag99} generalized the Kingman's coalescent to the
{\it $\La$-coalescent} which allows {\it multiple collisions}, i.e.,
more than two blocks may merge at a time. The $\La$-coalescent is
defined as a $\mathcal{P}_{\infty}$-valued Markov process $\{\Pi(t):
t\geq 0\}$ such that for each $n\in[\infty]$, $\{\Pi_n(t): t\geq
0\}$, its restriction to $[n]$,  is a $\mathcal{P}_n$-valued Markov
process whose transition rates are described as follows: if there
are currently $b$ blocks in the partition, then each $k$-tuple of
blocks ($2\leq k\leq b$) independently merges to form a single block
at rate
\begin{eqnarray}\label{la_rate}
\la_{b,k}=\int_0^1x^{k-2}(1-x)^{b-k}\La(dx),
\end{eqnarray}
where $\La$ is a finite measure on $[0,1]$. It is easy to check that
the rates $\(\la_{b,k}\)$ are consistent so that for all $2\leq
k\leq b$,
\begin{equation}
\begin{split}\label{20120607eq1}
\la_{b,k}=\la_{b+1,k}+\la_{b+1, k+1}.
\end{split}
\end{equation}
Consequently, for all $m<n<\infty$, the coalescent process
$R_m\(\Pi_n(t)\)$ given $\Pi_n(0)=\pi_n$ has the same distribution
as $\Pi_m(t)$ given $\Pi_m(0)=R_m(\pi_n)$.

With the transition rates determined by (\ref{la_rate}), there
exists an one to one correspondence between $\La$-coalescents and
finite measures $\La$ on $[0,1]$.

For $n=2,3,\ldots$, denote by
$$\la_n=\sum_{k=2}^n{n\choose k}\la_{n,k}$$  the total coalescence
rate starting with $n$ blocks. In addition, denote by
$$\gamma_n=\sum_{k=2}^n\(k-1\){n\choose k}\la_{n,k}$$
the rate at which the number of blocks decreases.
\subsection{Coming down from infinity}
For $n\in[\infty]$, let $\#\Pi_n(t)$ denote the number of blocks in
partition $\Pi_n(t)$. The $\La$-coalescent {\it comes down from
infinity} if $$P\(\#\Pi_{\infty}(t)<\infty\)=1$$ for all $t>0$. It
{\it stays infinite} if
$$P\(\#\Pi_{\infty}(t)=\infty\)=1$$ for all $t>0$.

\begin{theorem}[Schweinsberg \cite{Jason2000}]\label{comedown}
Suppose that $\La$ has no atom at $1$. Then
\begin{itemize}
\item the $\La$-coalescent comes down from infinity if and only if
$\sum_{b=2}^{\infty}\gamma_b^{-1}<\infty$;
\item the
$\La$-coalescent stays infinite if and only if
$\sum_{b=2}^{\infty}\gamma_b^{-1}=\infty$.
\end{itemize}
\end{theorem}

We list some examples of $\Lambda$-coalescents and
identify  whether they come down from infinity or stay infinite.
\begin{itemize}
\item If $\La=\delta_1$, then $\la_{b,b}=1$ and $\la_{b,k}=0$ for all $2\leq k\leq
b-1$. The corresponding coalescent only allows all the blocks to
merge into one single block after an exponential time with parameter
$1$. Thus it neither comes down from infinity nor stays infinite.
\item If $\La=\delta_0$, the corresponding coalescent degenerates to Kingman's
coalescent and comes down from infinity.
\item We say that a $\La$-coalescent has the $\(c,\epsilon,\gamma\)$-property, if
there exist constants $c>0$ and $\ep$, $\gamma\in(0,1)$ such that the measure $\La$ restricted to $[0,\ep] $  is absolutely continuous with respect to Lebesgue measure  and $${\La(dx)}\geq cx^{-\gamma}dx \text{\, for all \,} x\in [0,\ep].$$
The $\La$-coalescents with the $\(c,\epsilon,\gamma\)$-property come down from infinity.
\item For $\be\in(0, 2)$, the {\it Beta$(2-\be,\be)$-coalescent} is
the $\La$-coalescent with the finite measure $\La$ on $[0,1]$
denoted by
$$\La(dx)=\frac{\Gamma(2)}{\Gamma(2-\be)\Gamma(\be)}x^{1-\be}\(1-x\)^{\be-1}dx.$$
\begin{itemize}
\item If $\be\in\left(0,1\right]$, the Beta$(2-\be,\be)$-coalescent stays infinite.
\item If $\be\in\left(1,2\right)$, the
Beta$(2-\be,\be)$-coalescent has the $(c,\epsilon,\be-1)$-property and comes down from infinity.
\end{itemize}
\end{itemize}
\section{The $\La$-Fleming-Viot process and its lookdown construction}\label{s3}
In this section, we first briefly review the literatures on lookdown
construction. Then we introduce the lookdown construction for
$\La$-Fleming-Viot process and explain how to recover the
$\La$-coalescent from the lookdown construction.

\subsection{The lookdown construction for $\La$-Fleming-Viot process} The idea of expressing the probability-measure-valued process using
empirical measure of an exchangeable particle system goes back to
Dawson and Hochberg \cite{Dawson} in which the classical
Fleming-Viot process on space $E$ can be obtained as limit of
empirical measure of an $E^{\infty}$-valued particle system.

Donnelly and Kurtz \cite{DK96, DK99a, DK99b} explicited this idea
further by introducing the celebrated lookdown construction. In
\cite{DK96} they showed that the classical Fleming-Viot process
arises as limit of empirical measure associated with an infinite
particle system obtained from the lookdown scheme. In \cite{DK99a}
they proposed the lookdown representation for Fleming-Viot process
incorporating  selection and recombination. In \cite{DK99b} they
further introduced a modified lookdown construction for a large
class of measure-valued stochastic processes that include both the
neutral Fleming-Viot process and the Dawson-Watanabe superprocess as
examples.

Birkner and Blath \cite{BB} discussed the modified lookdown
construction of \cite{DK99b} for the $\La$-Fleming-Viot process.
They also described how to recover the $\La$-coalescent from their
modified lookdown construction. A Poisson point process construction
of the more general $\Xi$-lookdown model is found in Birkner et al.
\cite{MJM} by extending the modified lookdown construction of
Donnelly and Kurtz \cite{DK99b}. It was shown that the empirical
measure of the lookdown particle system converges almost surely on
the Skorokhod space of measure-valued paths to the $\Xi$-Fleming-Viot process.

In the lookdown particle system each particle is attached a ``level"
from the set $[\infty]$. There are reproduction events in the system
and the particles move independently between the reproduction times.
The system is constructed in a way that  the evolution of particle
at level $n$ only depends on the evolutions of particles at lower
levels. This property allows us to construct approximating particle
systems, and their limit as $n\rightarrow\infty$ in
the same probability space. The lookdown construction is a powerful
tool for studying Fleming-Viot processes.

Now we introduce the lookdown construction for
$\La$-Fleming-Viot process with underlying Brownian motion following Birkner and Blath \cite{BB}. Let
$$\(X_1(t),X_2(t),X_3(t),\ldots\)$$ be an $(\RR^d)^{\infty}$-valued random variable. For any $i\in[\infty]$,
$X_i(t)$ represents the spatial location of the particle at level
$i$. We require the initial values $\{ X_i(0),i\in[\infty]\}$ to be
exchangeable random variables so that
$$\lim_{n\rightarrow\infty}\frac{1}{n}\sum_{i=1}^n\delta_{X_i(0)}$$
exists almost surely by  de Finetti's theorem.

Let $\La$ be the finite measure associate to the $\La$-coalescent.
The reproduction in the particle system consists of two kinds of
birth events: the events of single birth  determined by measure
$\La(\{0\})\delta_0$ and the events of multiple births determined by
measure $\La$ restricted to $(0,1]$ that is denoted by $\La_0\equiv \Lambda-\La(\{0\})\delta_0$.

To describe the evolution of the system during events of single birth, let $\{\mathbf{N}_{ij}(t): 1\leq i<j<\infty\}$
be independent Poisson processes with rate $\La(\{0\})$. At a jump
time $t$ of $\mathbf{N}_{ij}$,
the particle at level $j$ looks down at the particle at level $i$ and
assumes its location (therefore, particle at level $i$ gives birth to a new particle). Values of  particles at levels
above $j$ are shifted accordingly, i.e., for $\Delta
{\mathbf{N}}_{ij}(t)=1$, we have
\begin{eqnarray}
X_k(t)=\begin{cases} X_k(t-), &\text{~~if $k<j$},\\
X_i(t-), &\text{~~if $k=j$},\\
X_{k-1}(t-), &\text{~~if $k>j$}.
\end{cases}
\end{eqnarray}

For those events of multiple births we can construct an independent
Poisson point process $\mathbf{N}$ on $\mathbb{R}^{+}\times\left(0,
1\right]$ with intensity measure $dt\otimes x^{-2}\La_0\(dx\)$. Let
$\{U_{ij}$, $i, j\in[\infty]\}$ be i.i.d. uniform $[0, 1]$ random
variables. Jump points $\{\(t_i, x_i\)\}$ for $\mathbf{N}$
correspond to the multiple birth events. For $J\subset [n]$ with
$|J|\geq 2$, define
\begin{eqnarray}\label{times}
\mathbf{N}_J^n(t)=\sum_{i:t_i\leq t}\prod_{j\in
J}\mathbf{1}_{\{U_{ij}\leq x_i\}}\prod_{j\in[n]\backslash  J}\mathbf{1}_{\{U_{ij}> x_i\}}.
\end{eqnarray}
Then ${\mathbf{N}}_J^n(t)$ counts the number of birth events, among
the levels $\{1,2,\ldots,n\}$, exactly those in $J$ were involved up
to time $t$. Intuitively, at a jump time $t_i$, a uniform coin is
tossed independently for each level. All the particles at levels $j$
with $U_{ij}\leq x_i$ participate in the lookdown event. More
precisely, those particles involved jump to the location of the
particle at the lowest level involved. The spatial locations of
particles on the other levels, keeping their original order, are
shifted upwards accordingly, i.e., if $t=t_i$ is the jump time and
$j$ is the lowest level involved, then
\begin{eqnarray*}
X_k(t)=
\begin{cases}
X_k(t-), \text{~for~} k\leq j,\\
X_j(t-), \text{~for~} k>j \text{~with~} U_{ik}\leq x_i,\\
X_{k-J_t^k}(t-), \text{~otherwise},
\end{cases}
\end{eqnarray*}
where $J_{t_i}^k\equiv \#\{m<k, U_{im}\leq x_i\}-1$.

Between jump times of the Poisson processes, particles at different levels move independently according to
Brownian motions. The values of $(X_i)$ are determined by the following system of
stochastic differential equations.
Let $\{\mathbf{B}_i(t): i=1,2,\ldots\}$ be a sequence of independent
and standard $d$-dimensional Brownian motions. The particle at level
$1$ evolves according to Brownian motion, i.e.,
\begin{eqnarray*}
X_1(t)=X_1(0)+ \mathbf{B}_1(t).
\end{eqnarray*}
For $n\geq 2$, $X_n$ satisfies
\begin{eqnarray*}
\begin{split}
X_n(t)=&X_n(0)+\mathbf{B}_n(t)+\sum_{1\leq
i<j<n}\int_0^t\(X_{n-1}(s-)-X_n(s-)\)d{\mathbf{N}}_{ij}(s)\\
&+ \sum_{1\leq i<n}\int_0^t\(X_i(s-)-X_{n}(s-)\)d{\mathbf{N}}_{in}(s)\\
&+\sum_{J\subset [n], n\in
J}\int_0^t\(X_{\min\(J\)}\(s-\)-X_n\(s-\)\)d\mathbf{N}_J^n\(s\)\\
&+\sum_{J\subset [n], n\not\in
J}\int_0^t\(X_{n-J_s^n}\(s-\)-X_n\(s-\)\)d\mathbf{N}_J^n\(s\),
\end{split}
\end{eqnarray*}
where the first integral concerns the lookdown event involving
levels $i$ and $j$ both below level $n$; the second integral
concerns the event  that particle at level $n$ looks down at particle at a lower level $i$; the third
and fourth integrals concern the
lookdown events with multiple levels involved.

For each $t>0$, $X_1(t),X_2(t),\ldots$ are known to be exchangeable random
variables so that
$$X(t)=\lim_{n\rightarrow\infty}\frac{1}{n}\sum_{i=1}^n\delta_{X_i(t)}$$
exists almost surely by  de Finetti's theorem and follows the probability law of the
$\La$-Fleming-Viot process with underlying Brownian motion (cf. Donnelly and
Kurtz \cite{DK99b}, Birkner and Blath \cite{BB}).

\subsection{The $\La$-coalescent recovered from the lookdown construction}\label{3.3}
The birth events induce a family structure to the particle system so we can talk about its genealogy.
For any $t\geq0$, $0\leq s\leq t$ and
$n\in[\infty]$, denote by $L_n^t(s)$  the ancestor's level at time $s$  for the particle with level $n$ at time $t$. Consequently,
$L_n^t(s)$ satisfies the equation:
\begin{eqnarray*}
\begin{split}
L_n^t(s)=n&-\sum_{1\leq i<j<n}\int_{s-}^t\mathbf{1}_{\{L_n^t(u)>j\}}d\mathbf{N}_{ij}(u)\\
&-\sum_{1\leq i<j\leq n}\int_{s-}^t\(j-i\)\mathbf{1}_{\{L_n^t(u)=j\}}d\mathbf{N}_{ij}(u)\\
&-\sum_{J\subset [n]}\int_{s-}^t\(L_n^t\(u\)-\min\(J\)\)\mathbf{1}_{\{L_n^t\(u\)\in
J\}}d\mathbf{N}_J^n(u)\\
&-\sum_{J\subset [n]}\int_{s-}^t\(|J\cap\{1,\ldots,L_n^t\(u\)\}|-1\)\times
\mathbf{1}_{\{L_n^t\(u\)>\min\(J\), L_n^t\(u\)\not\in
J\}}d\mathbf{N}_J^n\(u\).
\end{split}
\end{eqnarray*}
For fixed $T>0$ and $i\in[\infty]$, $L_i^T\(T-t\)$ represents the
ancestor's level at time $T-t$ of the particle with level $i$ at
time $T$ and $X_{L_i^T\(T-t\)}\((T-t)-\)$ represents that ancestor's
location.

Write $\big(\Pi(t)\big)_{0\leq t\leq T}$ for the
$\mathcal{P}_\infty$-valued process such that $i$ and $j$ belong to
the same block of $\Pi(t)$ if and only if $L_i^T(T-t)=L_j^T(T-t)$.
Therefore, $i$ and $j$ belong to the same block if and only if the
two particles at levels $i$ and $j$, respectively, at time $T$ share
a common ancestor at time $T-t$. The process $\big(\Pi(t)\big)_{ 0\leq t\leq T}$ turns out to have the same law as the
$\La$-coalescent running up to time $T$ (cf. Donnelly and Kurtz
\cite{DK99b}, Birkner and Blath \cite{BB}), which justifies the
usage of the notion.

We end this section with an observation on $\big(\Pi(t)\big)_{ 0\leq t\leq T}$.
\begin{lemma}\label{level}
For any fixed $T>0$, let $\big(\Pi(t)\big)_{0\leq t\leq T}$ be the
$\La$-coalescent recovered from the lookdown construction. Then
given $t\in[0, T]$ and the ordered random partition
$\Pi(t)=\{\pi_l(t): l=1,\ldots,\#\Pi(t)\}$, we have
$$L_j^T\(T-t\)=l\text{~~for any~~}j\in\pi_l(t).$$
\end{lemma}
\begin{proof}
For any $1\leq l\leq \#\Pi(t)$, by definition the particles with levels in block
$\pi_l(t)$ at time $T$ have the same ancestor at time $T-t$. Let
$i_l=\min\pi_l(t)$.

It is trivial that $i_1=1$ and $L_{i_1}^T(T-t)=1$. Then
$L_{j}^T(T-t)=1$ for all $ j\in\pi_1(t)$.

Now we consider the case $l\geq 2$. Since $i_l=\min\pi_l(t)$, looking
forwards in time the ancestor of the particle on level $i_l$ at time
$T$ never looks down to particles of lower levels during the time
interval $[T-t,T]$. As an increasing and piecewise constant
function, the ancestor level $\{L_{i_l}^T(s), T-t\leq s\leq T\}$
only increases in $s$  because of upward shifts. We thus have
\begin{eqnarray}\label{1202071}
\begin{split}
&L_{i_l}^T(T)-L_{i_l}^T(T-t)
=\#\left\{L_j^T(T):j\in[i_l]\right\}-\#\left\{L_j^T(T-t):j\in[i_l]\right\},
\end{split}
\end{eqnarray}
where we recall that $[i_l]=\left\{1,2,\ldots,i_l\right\}$. By
(\ref{1202071}), it follows that $$L_{i_l}^T(T-t)=
\#\left\{L_j^T(T-t):j\in[i_l]\right\}.$$ Since $\{{\pi}_j(t), 1\leq
j\leq\#\Pi(t)\}$ are ordered by their minimal elements, then
\[[i_l]\subset \cup_{j=1}^l\pi_j(t)  \text{\,\, and \,\,} \pi_j(t)\cap [i_l]\neq\emptyset\text{\,\, for all \,\,} 1\leq j \leq l. \]
Recall that $\Pi_{i_l}(t)$ is the restriction of
$\Pi(t)$ to $[i_l]$. It follows that
$\#\Pi_{i_l}(t)=l$,
which implies that
$$\#\left\{L_j^T(T-t):j\in[i_l]\right\}=l.$$
We then have $L_{i_l}^T(T-t)=l$. All the particles with levels from
the same block have a common ancestor at time $T-t$, therefore,
$L_j^T(T-t)=l$ for any
$j\in\pi_l(t)$.
\end{proof}

\section{The compact support property for the $\Lambda$-Fleming-Viot
process}\label{s5} In this section, we proceed to prove the compact support
property of $\La$-Fleming-Viot process. We also find both  upper bound
and lower bound for the Hausdorff dimension of its support. Until the end of this section we assume that
the measure $\Lambda$ has no mass at $1$, i.e., $\Lambda(\{1\})=0 $.

Intuitively, if the corresponding $\La$-coalescent comes
down from infinity, then for any fixed $T>0$, the random variables
 $\(X_1(T), X_2(T),\ldots\)$ in the lookdown system are
highly correlated. This is because the particles at time $T$ are
offspring of the finitely many particles alive at an arbitrary time
before $T$. Our approach is to group the countably many particles at
time $T$ into finitely many disjoint subclusters according to their
respective common ancestors at an earlier time. When this earlier
time is close enough to $T$, the distances between the particles at
time $T$ and their respective ancestors have to be small. Then each
subcluster is contained in a small neighborhood of its ancestor and
all the neighborhoods are contained in a compact set. The compact
support property thus follows.

Throughout this paper we always write $C$ or $C$ with subscript for a positive constant and write $C(y)$ for a positive constant depending on $y$, whose values might vary from place to
place.

\subsection{The main result}
We begin by recalling the notion of {\it Hausdorff dimension}.
Given $A\subset\mathbb{R}^d$ and $\be>0$, $\eta>0$,
let
$${\bf \La}_{\eta}^{\be}\(A\)\equiv\inf_{\{S_l\}\in{\varphi}_{\eta}}\sum_l\(d\(S_l\)\)^{\be},$$
where $d\(S_l\)$ denotes the diameter of ball $S_l$ in
$\mathbb{R}^d$ and $\varphi_{\eta}$ denotes the collection of {\it
$\eta$-covers} of set $A$ by balls, i.e.,
$$\varphi_{\eta}\equiv\big\{\{S_l\} \text{ is a cover of $A$ by balls with }
d\(S_l\)<\eta\text{~for each~} l\big\}.$$ The Hausdorff
$\be$-measure of $A$ is defined by $${\bf
\La}^{\be}\(A\)=\lim_{\eta\rightarrow 0}{\bf
\La}_{\eta}^{\be}\(A\).$$ The Hausdorff dimension of
$A$ is defined by
$$\text{dim~}(A)\equiv\inf\big\{\be>0: {\bf
\La}^{\be}\(A\)=0\big\}=\sup\big\{\be>0: {\bf
\La}^{\be}\(A\)=\infty\big\}.$$

For any $n>m\geq 2$, let $\Pi^\Lambda_n(t), t\geq 0$, be any
$\La$-coalescent with $\Pi_n^\Lambda(0)=\mathbf{0}_{[n]}$. Then the
block counting process $\#\Pi^\Lambda_n(t)\vee m$ is a Markov chain
with initial value $n$ and absorbing state $m$. For any $n\geq b>m$,
let $\{\mu_{b,k}\}_{m\leq k\leq b-1}$  be its transition rates such
that
\begin{eqnarray}\label{M_rate}
\begin{cases}
&\mu_{b,b-1}={b\choose 2}\la_{b,2},\\
&\mu_{b,b-2}={b\choose 3}\la_{b,3},\\
&\cdots\cdots\\
&\mu_{b,m+1}={b\choose b-m}\la_{b,b-m},\\
&\mu_{b,m}=\sum_{k=b-m+1}^b{b\choose k}\la_{b,k}.
\end{cases}
\end{eqnarray}
The total transition rate is
$$\mu_b=\sum_{k=m}^{b-1}\mu_{b,k}=\sum_{k=2}^b{b\choose
k}\la_{b,k}=\la_b.$$

For $b>m$, let $\gamma_{b,m}$ be the total rate at which the block counting Markov chain  starting at $b$ is decreasing. Then
\begin{eqnarray}\label{1213}
\gamma_{b,m}=\begin{cases}\sum_{k=2}^{b-m}\(k-1\){b\choose
k}\la_{b,k}+\sum_{k=b-m+1}^{b}\(b-m\){b\choose k}\la_{b,k},
\text{~~if~~} b\geq m+2,\\
\sum_{k=2}^{b}{b\choose k}\la_{b,k}, \text{~~if~~} b=m+1.
\end{cases}
\end{eqnarray}

The main result of this paper is the theorem below, which gives a
sufficient condition for the $\La$-Fleming-Viot process to have a
compact support.

\begin{theorem}\label{compact}
Given any $\Lambda$-Fleming-Viot process $X$ with underlying Brownian motion in $\RR^d$, let $(\gamma_{b,m})$ be defined in Equation
(\ref{1213}) for the corresponding $\Lambda$-coalescent.
If there exist constants $C>0$ and $\al>0$ such that
\begin{eqnarray}\label{VI}
\sum_{b=m+1}^{\infty}{\gamma_{b,m}}^{-1}\leq{C}{m^{-\al}}
\end{eqnarray}
for $m$ big enough, then for any $T>0$, with probability one the random
measure $X(T)$ has a compact support and the Hausdorff dimension for
supp$X(T)$ is bounded from above by $ 2/\al$.
\end{theorem}

\subsection{Some estimates on the $\La$-coalescent}
We first point out an immediate consequence of assumption (\ref{VI}).
\begin{lemma}\label{come_down}
 The $\La$-coalescent comes down from infinity under assumption (\ref{VI}).
\end{lemma}
\begin{proof}
Notice that $\gamma_{b,m}\leq\gamma_b$ for any $b>m\geq 2$. Then the
corresponding $\La$-coalescent comes down from infinity by Theorem
\ref{comedown}.
\end{proof}

\begin{lemma}\label{1202}
For any $2\leq  m<b$, we have $\gamma_{b,m}\leq\gamma_{b+1,m}$.
\end{lemma}
\begin{proof}
According to the different values of $b$ and $m$, we consider the
following three different
cases separately.

\noindent Case I: $b=m+1$. By the consistency condition
(\ref{20120607eq1}) for $\(\la_{b,k}\)$ and the definition of
$(\gamma_{b,m})$, we have
\begin{eqnarray*}
\begin{split}
\gamma_{b,m}=&\sum_{k=2}^b{b\choose k}\la_{b,k}
=\sum_{k=2}^b{b\choose k}\(\la_{b+1,k}+\la_{b+1,k+1}\)\\
=&{b\choose 2}\la_{b+1,2}+\sum_{k=3}^b{b\choose
k}\la_{b+1,k}+\sum_{k=3}^b{b\choose k-1}\la_{b+1,k}+\la_{b+1,b+1}\\
=&{b\choose 2}\la_{b+1,2}+\sum_{k=3}^b\({b\choose
k}+{b\choose k-1}\)\la_{b+1,k}+\la_{b+1,b+1}.\\
\end{split}
\end{eqnarray*}
By the identity
\begin{equation}\label{20120607eq2}
{n\choose k}+{n\choose k-1}={n+1 \choose k},
\end{equation}
we then have
\begin{eqnarray*}
\begin{split}
\gamma_{b,m}\leq&{b+1\choose 2}\la_{b+1,2}+\sum_{k=3}^b2{b+1\choose
k}\la_{b+1,k}+2\la_{b+1,b+1}=\gamma_{b+1,m}.
\end{split}
\end{eqnarray*}

\noindent Case II: $b=m+2$. Similarly, it follows from  the consistency condition
(\ref{20120607eq1}) for $\(\la_{b,k}\)$ and the definition of
$(\gamma_{b,m})$ that
\begin{eqnarray*}
\begin{split}
\gamma_{b,m}=&{b\choose 2}\la_{b,2}+\sum_{k=3}^b2{b\choose k}\la_{b,k}\\
=&{b\choose 2}\la_{b+1,2}+{b\choose
2}\la_{b+1,3}+\sum_{k=3}^b2{b\choose
k}\(\la_{b+1,k}+\la_{b+1,k+1}\)\\
=&{b\choose 2}\la_{b+1,2}+{b\choose
2}\la_{b+1,3}+\sum_{k=3}^b2{b\choose
k}\la_{b+1,k}+\sum_{k=4}^b2{b\choose k-1}\la_{b+1,k}+2\la_{b+1,b+1}\\
\leq&{b+1\choose 2}\la_{b+1,2}+2\({b\choose 2}+{b\choose
3}\)\la_{b+1,3}+\sum_{k=4}^{b}3\({b\choose k}+{b\choose
k-1}\)\la_{b+1,k}+3\la_{b+1,b+1}.
\end{split}
\end{eqnarray*}
Applying the identity (\ref{20120607eq2}), we have
\begin{eqnarray*}
\begin{split}
\gamma_{b,m}\leq&{b+1\choose 2}\la_{b+1,2}+2{b+1\choose
3}\la_{b+1,3}+\sum_{k=4}^b3{b+1\choose
k}\la_{b+1,k}+3\la_{b+1,b+1}\\
=&\sum_{k=2}^3\(k-1\){b+1\choose
k}\la_{b+1,k}+\sum_{k=4}^{b+1}3{b+1\choose k}\la_{b+1,k}\\
=&\gamma_{b+1,m}.
\end{split}
\end{eqnarray*}

\noindent Case III: $b\geq m+3$. The proof involves similar but longer arguments
as the first two cases.
\begin{eqnarray*}
\begin{split}
\gamma_{b,m}=&\sum_{k=2}^{b-m}\(k-1\){b\choose
k}\la_{b,k}+\sum_{k=b-m+1}^b\(b-m\){b\choose
k}\la_{b,k}\\
=&\sum_{k=2}^{b-m}\(k-1\){b\choose
k}\(\la_{b+1,k}+\la_{b+1,k+1}\)+\sum_{k=b-m+1}^b\(b-m\){b\choose
k}\(\la_{b+1,k}+\la_{b+1,k+1}\)\\
=&\sum_{k=2}^{b-m}\(k-1\){b\choose
k}\la_{b+1,k}+\sum_{k=b-m+1}^b\(b-m\){b\choose
k}\la_{b+1,k}\\
&+\sum_{k=3}^{b-m+1}\(k-2\){b\choose
k-1}\la_{b+1,k}+\sum_{k=b-m+2}^{b+1}\(b-m\){b\choose
k-1}\la_{b+1,k}\\
=&{b\choose 2}\la_{b+1,2}+\sum_{k=3}^{b-m}\(\(k-1\){b\choose
k}+\(k-2\){b\choose k-1}\)\la_{b+1,k}\\
&+\(b-m-1\){b\choose b-m}\la_{b+1,b+1-m}+\(b-m\){b\choose
b-m+1}\la_{b+1,b-m+1}\\
&+\sum_{k=b-m+2}^b\(b-m\)\({b\choose
k-1}+{b\choose k}\)\la_{b+1,k}+(b-m)\la_{b+1,b+1}.\\
\end{split}
\end{eqnarray*}
With Equation (\ref{20120607eq2}),
\begin{eqnarray*}
\begin{split}
\gamma_{b,m}\leq&{b\choose 2}\la_{b+1,2}+\sum_{k=3}^{b-m}\(k-1\){b+1
\choose k}\la_{b+1,k}+\(b-m\){b+1 \choose
b+1-m}\la_{b+1,b+1-m}\\
&+\sum_{k=b-m+2}^b\(b-m\){b+1\choose
k}\la_{b+1,k}+(b-m)\la_{b+1,b+1}\\
\leq&\sum_{k=2}^{b+1-m}\(k-1\){b+1\choose
k}\la_{b+1,k}+\sum_{k=b+2-m}^{b+1}\(b+1-m\){b+1 \choose k}\la_{b+1,k}\\
=&\gamma_{b+1,m}.
\end{split}
\end{eqnarray*}
\end{proof}
For the $\La$-coalescent $\left(\Pi(t)\right)_{0\leq t\leq T}$ with
$\Pi(0)=\mathbf{0}_{[\infty]}$ recovered from the lookdown
construction, denote by $\left(\Pi_n(t)\right)_{ 0\leq t\leq T}$ its
restriction to $[n]$. Clearly $\left(\Pi_n(t)\right)_{ 0\leq t\leq T}$ is
exactly the $\La$-coalescent recovered from the first $n$ levels of
the lookdown construction.

For any $n>m\geq 2$, put
$$T^n_m=\inf\big\{t\in[0,T]: \#\Pi_n(t)\leq m \big\}$$ and
\begin{equation}\label{T_m}
T_m\equiv T_m^{\infty}=\inf\big\{t\in[0,T]: \#\Pi(t)\leq m\big\}
\end{equation}
with the convention  $\inf\emptyset=T$. From the lookdown
construction, it is obvious that
\begin{eqnarray}\label{time_increase}
T^n_m\leq T^{n+1}_m\leq T^{n+2}_m\leq\cdots\leq \uparrow T_m.
\end{eqnarray}

\begin{lemma}\label{time}
If there exist constants $C>0$ and $\al>0$ satisfying (\ref{VI}) for
the corresponding $\La$-coalescent, then we have
\begin{eqnarray}
\EE T_m\leq{C}{m^{-\al}}
\end{eqnarray}
for $m$ big enough.
\end{lemma}
\begin{proof}
We can define a $\Lambda$-coalescent $(\Pi^\Lambda(t))_{t\geq 0}$ satisfying  $\Pi^\Lambda(t)=\Pi(t) $ for all $t\leq T$.
Put
$$T^{\Lambda,n}_m=\inf\big\{t\geq 0: \#\Pi^\Lambda_n(t)\leq m \big\}$$ and
$$T^\Lambda_m\equiv T_m^{\Lambda,\infty}=\inf\big\{t\geq 0: \#\Pi^\Lambda(t)\leq m\big\}$$
with the convention  $\inf\emptyset=\infty$.  Then $T_m\leq
T^\Lambda_m$ and we only need to show that $\EE
T_m^\Lambda\leq{C}{m^{-\al}} $.

We adapt the idea of Lemma 6 in Schweinsberg \cite{Jason2000} to
prove this lemma. For any $n>m$ and $1\leq k\leq n-m$, define
\begin{eqnarray*}
\begin{split}
&\mathcal{R}_0=0,\\
&\mathcal{R}_k=\begin{cases}\inf\big\{t\geq 0:
\#\Pi_n^\Lambda(t)<\#\Pi_n^\Lambda(\mathcal{R}_{k-1})\big\}
&\text{\ \ \ \ if  $\#\Pi_n^\Lambda(\mathcal{R}_{k-1})> m$},\\
\mathcal{R}_{k-1} &\text{\ \ \ \ if
$\#\Pi_n^\Lambda(\mathcal{R}_{k-1})=m$}.
\end{cases}
\end{split}
\end{eqnarray*}
Note that $T^{\Lambda,n}_m=\mathcal{R}_{n-m}$. For $i=0,1,2,\ldots,n-m$, let
$\mathcal{N}_i=\#\Pi_n^\Lambda(\mathcal{R}_i)$. For
$i=1,2,\ldots,n-m$, let $L_i=\mathcal{R}_i-\mathcal{R}_{i-1}$ and
$\mathcal{J}_i=\mathcal{N}_{i-1}-\mathcal{N}_i$.

On the event $\{\mathcal{N}_{i-1}>m\}$, for any $n\geq b>m$, we have
$$P\(\mathcal{J}_i=k-1|\mathcal{N}_{i-1}=b\)={b\choose k}\frac{\la_{b,k}}{\la_b}$$
for $k=2,3,\ldots,b-m$ and
$$P\(\mathcal{J}_i=b-m|\mathcal{N}_{i-1}=b\)=\sum_{k=b-m+1}^b{b\choose k}\frac{\la_{b,k}}{\la_b}.$$
Consequently, on the event $\{\mathcal{N}_{i-1}>m\}$, we have
$$\EE\(\mathcal{J}_i|\mathcal{N}_{i-1}=b\)=\sum_{k=2}^{b-m}\(k-1\){b\choose k}\frac{\la_{b,k}}{\la_b}+\(b-m\)\sum_{k=b-m+1}^b{b\choose k}\frac{\la_{b,k}}{\la_b}
=\frac{\gamma_{b,m}}{\la_b}.$$
Therefore,
\begin{eqnarray*}
\begin{split}
\EE T^{\Lambda,n}_m&=\EE \mathcal{R}_{n-m}=\EE
\sum_{i=1}^{n-m}L_i=\sum_{i=1}^{n-m}\EE\EE\(L_i|\mathcal{N}_{i-1}\)
=\sum_{i=1}^{n-m}\EE\(\la_{\mathcal{N}_{i-1}}^{-1}\mathbf{1}_{\{\mathcal{N}_{i-1}>m\}}\)\\
&=\sum_{i=1}^{n-m}\EE\(\gamma_{\mathcal{N}_{i-1},m}^{-1}\EE\(\mathcal{J}_i|\mathcal{N}_{i-1}\)\mathbf{1}_{\{\mathcal{N}_{i-1}>m\}}\)\\
&=\sum_{i=1}^{n-m}\EE\EE\(\gamma_{\mathcal{N}_{i-1},m}^{-1}\mathcal{J}_i\mathbf{1}_{\{\mathcal{N}_{i-1}>m\}}|\mathcal{N}_{i-1}\).
\end{split}
\end{eqnarray*}
Since $\mathcal{J}_i=0$ on the event $\{\mathcal{N}_{i-1}=m\}$, we
have
\begin{eqnarray*}
\begin{split}
\EE
T^{\Lambda,n}_m&=\sum_{i=1}^{n-m}\EE\(\EE\(\gamma_{\mathcal{N}_{i-1},m}^{-1}\mathcal{J}_i|\mathcal{N}_{i-1}\)\)
=\sum_{i=1}^{n-m}\EE\(\gamma_{\mathcal{N}_{i-1},m}^{-1}\mathcal{J}_i\)\\
&=\EE\(\sum_{i=1}^{n-m}\gamma_{\mathcal{N}_{i-1},m}^{-1}\mathcal{J}_i\)
=\EE\(\sum_{i=1}^{n-m}\sum_{j=0}^{\mathcal{J}_i-1}\gamma_{\mathcal{N}_{i-1},m}^{-1}\).
\end{split}
\end{eqnarray*}
Since $\(\gamma_{b,m}\)_{b=m+1}^{\infty}$ is an increasing sequence
by Lemma \ref{1202}, it follows that
\begin{eqnarray*}
\begin{split}
\EE
T^{\Lambda,n}_m\leq&\EE\(\sum_{i=1}^{n-m}\sum_{j=0}^{\mathcal{J}_i-1}\gamma_{\mathcal{N}_{i-1}-j,m}^{-1}\)
=\EE\(\sum_{b=m+1}^n\gamma_{b,m}^{-1}\)\leq\sum_{b=m+1}^{\infty}\gamma_{b,m}^{-1}.
\end{split}
\end{eqnarray*}
By the Monotone Convergence Theorem, we have
$$\EE T^\Lambda_m=\lim_{n\rightarrow\infty}\EE T^{\Lambda,n}_m\leq\sum_{b=m+1}^{\infty}\gamma_{b,m}^{-1}.$$
Finally, by (\ref{VI}) we have
$$\EE T^\Lambda_m\leq{C}{m^{-\al}}$$
for $m$ big enough.
\end{proof}

\subsection{An estimate on standard Brownian motion}
Write
$$\(\mathbf{B}(s)\)_{ s\geq0}=\(B_1(s),B_2(s),\ldots,B_d(s)\)_{s\geq 0}$$ for standard
$d$-dimensional Brownian motion with initial value $\mathbf{0}$,
where $$\(B_i(s)\)_{s\geq0},~~~~ i=1,\dots,d$$ are independent
one-dimensional standard Brownian motions. For any vector
$\mathbf{z}=\(z_1,z_2,\ldots,z_d\)\in\mathbb{R}^d$, write
$\|\mathbf{z}\|=\sqrt{\sum_{i=1}^dz_i^2}$ as usual.

\begin{lemma}\label{0108}
Given the above mentioned $d$-dimensional standard Brownian motion,
for any $t>0$ and $x>0$, there exist positive constants $C_1$ and
$C_2$ such that
\begin{eqnarray}
P\(\sup_{\begin{smallmatrix}0\leq s\leq
t\end{smallmatrix}}\|\mathbf{B}(s)\|> x\)\leq
C_1{\sqrt{t}}x^{-1}\exp\(-{C_2x^2}{t^{-1}}\).
\end{eqnarray}
\end{lemma}
\begin{proof}
By the {\it reflection principle}, it is clear that
\begin{eqnarray}\label{Brownian}
\begin{split}
P\(\sup_{0\leq s\leq t}\|\mathbf{B}(s)\|> x\)
&\leq 2dP\(|B_1(t)|> x/\sqrt{d}\)\\
&\leq
\(8d^3/\pi\)^{1/2}\sqrt{t}x^{-1}\exp\(-\frac{x^2}{2dt}\)\\
&\equiv C_1\sqrt{t}x^{-1}\exp\(-{C_2x^2}{t^{-1}}\).
\end{split}
\end{eqnarray}
\end{proof}

\subsection{The compact support property for the $\La$-Fleming-Viot process}\label{1207}
In this section, we discuss the $\La$-Fleming-Viot process with the
corresponding coalescent satisfying assumption (\ref{VI}). By Lemma
\ref{come_down}, the corresponding coalescent comes down from
infinity. Given the constant $\al>0$ in (\ref{VI}), for any
$k\in[\infty]$ define
$$N_k=2^{{k}/{\al}}k^{{2}/{\al}}.$$

Because of the coming down from infinity property, if we look
backwards for a small amount of positive time $\Delta T$ in the lookdown construction, there exist only
finitely many ancestors at time $T-\Delta T$ whose offspring are
these countably many particles existing at time $T$.

Recall the $\La$-coalescent $\left(\Pi(t) \right)_{0\leq t\leq T}$
recovered from the lookdown construction in Section \ref{3.3} and
the time $T_m$ defined by (\ref{T_m}). For all $m\in[\infty]$, the
number of ancestors at time $T- T_{N_m}$ is equal to
$\#\Pi(T_{N_m})$, which is almost surely finite by the coming down from infinity property.

Put
$$N_m^*=\#\Pi\(T_{N_m}\) \text{\,\,\, and \,\,\,}
\Pi\(T_{N_m}\)=\big\{\pi_l: l=1,\ldots,N_m^{*}\big\},$$ where
$\{{\pi}_l\equiv{\pi}_l(m), l\in[N_m^*]\}$ are all the disjoint
blocks of $\Pi\( T_{N_m}\)$ ordered by their minimal elements. Note
that $L_j^T(T-T_{N_m})=l $ for any $j\in\pi_l$ by Lemma \ref{level}.
The {\it maximal radius of subclusters} is defined as:
\begin{equation*}
\begin{split}
R_m
&\equiv\max_{1\leq l\leq N_m^*}\sup_{j\in\pi_l}\left\|X_j(T)-X_{L_{j}^T(T-T_{N_m})}\((T- T_{N_m})-\)\right\|\\
&=\max_{1\leq l\leq
N_m^*}\sup_{j\in\pi_l}\left\|X_j(T)-X_l\((T-
T_{N_m})-\)\right\|.
\end{split}
\end{equation*}
For $k\in [\infty]$, define time interval ${J}_k=\left[T- T_{N_k}, T-
T_{N_{k+1}}\right]$. Let $|J_k|$ be the length of  interval
$J_k$. Thus
$$|J_k|=\(T- T_{N_{k+1}}\)-\(T- T_{N_{k}}\)= T_{N_{k}}
- T_{N_{k+1}}\leq T_{N_k}.$$ Let $D_{k}$ be the maximal
dislocation over time interval $J_k$ of all the Brownian motions
involved, i.e.,
\begin{equation}\label{12229}
D_{k}\equiv\max_{1\leq l\leq N_k^*}\sup_{j\in
\pi_l}\left\|X_{L^T_j(T- T_{N_{k+1}})}(T-
T_{N_{k+1}})-X_{l}\((T- T_{N_k})-\)\right\|.
\end{equation}
Note that for any fixed $1\leq l\leq N_k^*$, the collection of
ancestor levels $$\left\{L^T_j(T- T_{N_{k+1}}):
j\in\pi_l\right\}$$ has a finite cardinality because of the coming
down from infinity property. Thus the supremum in (\ref{12229}) is
taken over a finite set.

\begin{lemma}\label{110902}
Under the condition of Theorem \ref{compact}, for any
$\delta\in(0,{1}/{2})$, almost surely the maximal dislocation
$D_{k}$ satisfies
\begin{eqnarray*}
D_{k}\leq 2^{-k\(\frac{1}{2}-\delta\)}
\end{eqnarray*}
for $k$ big enough.
\end{lemma}
\begin{proof}
For the trivial case of $ T_{N_{k+1}}=T$, we have $|J_k|=0$ and the
dislocation of Brownian motion over $J_k$ is equal to $0$. Consequently,
$$P\(D_{k}>2^{-k\(\frac{1}{2}-\delta\)},|J_k|=0\)=0.$$
In the case of $|J_k|>0$, the total number of Brownian motions
involved over $J_k$ is no more than
$$N_{k+1}=2^{{\(k+1\)}/{\al}}\(k+1\)^{{2}/{\al}}.$$
Thus we have
\begin{eqnarray*}
\begin{split}
&P\(D_{k}>2^{-k\(\frac{1}{2}-\delta\)}\)\\
=&P\(D_{k}>2^{-k\(\frac{1}{2}-\delta\)},|J_k|=0\)
+P\(D_{k}>2^{-k\(\frac{1}{2}-\delta\)},|J_k|>0\)\\
=&P\(D_{k}>2^{-k\(\frac{1}{2}-\delta\)},0<|J_k|\leq 2^{-k}\)
+P\(D_{k}>2^{-k\(\frac{1}{2}-\delta\)}\Big||J_k|>2^{-k}\)
\times P\(|J_k|>2^{-k}\)\\
\leq&N_{k+1}\times P\(\sup_{\begin{smallmatrix}0\leq s\leq
2^{-k}\end{smallmatrix}}\|\mathbf{B}(s)\|>2^{-k\(\frac{1}{2}-\delta\)}\)
+P\(|J_k|>2^{-k}\)\\
\equiv&I_1(k)+I_2(k).
\end{split}
\end{eqnarray*}
By Lemma \ref{0108} we have
\begin{eqnarray*}
\begin{split}
&P\(\sup_{\begin{smallmatrix}0\leq s\leq
2^{-k}\end{smallmatrix}}\|\mathbf{B}(s)\|>2^{-k\(\frac{1}{2}-\delta\)}\)
\leq C_12^{-k\delta}\exp\(-{{C_2}2^{2\delta k}}\).
\end{split}
\end{eqnarray*}
Consequently,
\begin{eqnarray*}
\begin{split}
I_1(k) \leq&2^{\frac{k+1}{\al}}\(k+1\)^{\frac
{2}{\al}}C_12^{-k\delta}\exp\(-{{C_2}2^{2\delta
k}}\)\\
\leq &C_12^{\frac{4k}{\al}}\exp\(-{{C_2}2^{2\delta k}}\).\\
\end{split}
\end{eqnarray*}
It is clear that $\sum_{k}I_1(k)<\infty$.

Applying Lemma \ref{time} with $m$ replaced by
$N_k$, we have for $k$ large enough
$$\EE T_{N_k}\leq{C}2^{-{k}}k^{-2}.$$
Since $|J_k|\leq T_{N_k}$,  it
follows from the Markov's inequality that for $k$ large
\begin{equation*}
\begin{split}
I_2(k)\leq&P\(T_{N_k}>2^{-k}\) \leq {2^{k}}{\EE T_{N_k}}
\leq{C}k^{-2}.
\end{split}\end{equation*}
Therefore,
\begin{eqnarray*}
&&\sum_{k}P\(D_{k}>2^{-k\(\frac{1}{2}-\delta\)}\)
\leq\sum_{k}I_1(k)+\sum_{k}I_2(k)
<\infty.
\end{eqnarray*}
Applying the Borel-Cantelli lemma, we have almost surely
\begin{eqnarray*}
D_{k}\leq2^{-k\(\frac{1}{2}-\delta\)}
\end{eqnarray*}
for $k$ large enough.
\end{proof}
\begin{lemma}\label{110904}
Under the condition of Theorem \ref{compact}, for any
$\delta\in(0,{1}/{2})$, there exists a positive constant $C(\delta)$
such that almost surely,
$$R_{m}\leq C(\delta)2^{-m\(\frac{1}{2}-\delta\)}$$
 for $m$ big enough.
\end{lemma}
\begin{proof}
Applying Lemma \ref{110902}, we have almost surely,
\begin{eqnarray*}
D_{k}\leq2^{-k\(\frac{1}{2}-\delta\)}
\end{eqnarray*}
for $k$ large enough. Then almost surely for $m$ large enough,
\begin{eqnarray*}
\begin{split}
R_{m}&\leq\sum_{k=m}^{\infty}D_{k}
\leq\sum_{k=m}^{\infty}2^{-k\(\frac{1}{2}-\delta\)}\\
&=\frac{2^{-m\(\frac{1}{2}-\delta\)}}
{1-2^{-\(\frac{1}{2}-\delta\)}}
\equiv C(\delta)2^{-m\(\frac{1}{2}-\delta\)}.
\end{split}
\end{eqnarray*}
\end{proof}

\begin{proof} [{\bf Proof of Theorem \ref{compact}}]
We first prove the compact support property for the $\La$-Fleming-Viot process with the corresponding coalescent
satisfying (\ref{VI}) at any fixed time $T$.

For $m$ large enough and for all $k\geq m$, by Lemma \ref{110904} we have
\begin{eqnarray*}
\begin{split}
X_j\(\(T-
T_{N_k}\)-\)&\subseteq\bigcup_{l=1}^{N_m^{*}}{\mathbb{B}}\(X_l\(\(T-
T_{N_{m}}\)-\),R_m\)\\
&\subseteq\bigcup_{l=1}^{N_m^{*}}{\mathbb{B}}\(X_l\(\(T-
T_{N_{m}}\)-\),C(\delta)2^{-m\(\frac{1}{2}-\delta\)}\)\\
&\equiv \mathbf{B},
\end{split}
\end{eqnarray*}
where ${\mathbb{B}}(x,r)$ denotes the closed ball centered at $x$
with radius $r$. For each  $n\in[\infty]$, from the lookdown
construction there exists a random variable $\delta_n>0$ such that
during the time interval $[T-\delta_n, T]$, the particle at level
$n$ never looks down to those particles at lower levels
$\{1,2,\ldots,n-1\}$.
It then follows from Lemma \ref{level} that for any $j\in[n]$,
$L_j^T(s)=j$ for all $s\in[T-\delta_n, T]$. Further, the sample path
continuity for Brownian motion implies that
$$X_j(T)=X_j(T-)=\lim_{k\rightarrow\infty}X_j\(\(T- T_{N_k}\)-\).$$

Therefore, $X_j(T)$ is a limit point for the compact set
$\mathbf{B}$ and we have $X_j(T)\in\mathbf{B}$ for all $j$. Let
$$\hat{X}_n(T)\equiv\frac{1}{n}\sum_{i=1}^n\delta_{X_i(T)}.$$
By the lookdown construction for the $\La$-Fleming-Viot process we
have $$X(T)=\lim_{n\rightarrow\infty}\hat{X}_n(T).$$ Clearly,
$$\text{supp}\(\hat{X}_n(T)\)\subseteq\mathbf{B}$$
for all $n$, which implies that
$$\text{supp}\(X\(T\)\)\subseteq\mathbf{B}.$$

We now consider the Hausdorff dimension for the support at time $T$. The
collection of closed balls \[\left\{\mathbb{B}\(X_l\(\(T-
T_{N_{m}}\)-\),C(\delta)2^{-m\(\frac{1}{2}-\delta\)}\): l=1, \ldots,
N_m^{*}\right\}\] is a cover of $\text{supp}(X(T))$ for $m$ large
enough.

For any $\epsilon>0$,  choose $\delta>0$ small enough so that
\[\(\frac{1}{2}-\delta\)(2+\epsilon)>1.\]
For all $m$ big enough we also have $N_m^{*}\leq N_m$. Then
\begin{equation*}
\begin{split}
&\lim_{m\goto\infty}N_m^* C(\delta)^{\frac{2+\epsilon}{\al}}
2^{-m\(\frac{1}{2}-\delta\)\frac{2+\epsilon}{\al}}\\
\leq&
\lim_{m\goto\infty}
2^{\frac{m}{\al}}m^{\frac{2}{\al}}C(\delta)^{\frac{2+\epsilon}{\al}}2^{-m\(\frac{1}{2}-\delta\)\frac{2+\epsilon}{\al}}\\
=&C(\delta)^{\frac{2+\epsilon}{\al}}\lim_{m\goto\infty}
m^{\frac{2}{\al}}2^{-\frac{m}{\al}\left[\(\frac{1}{2}-\delta\)(2+\epsilon)-1\right]}\\
=&0
\end{split}
\end{equation*}
and we have
\[\text{dim} \left(\text{supp}(X(T))\right)\leq {\(2+\epsilon\)}/{\al}.\]
$\epsilon$ is arbitrary, so  the Hausdorff
dimension for the support is bounded from above by $2/\al$.
\end{proof}

\begin{corollary}\label{120111}
If there exist constants $C>0$ and $\al>0$ such that the total
coalescence rates $\(\la_b\)_{b\geq 2}$ of the corresponding
$\La$-coalescent $\left(\Pi(t) \right)_{0\leq t\leq T}$ satisfy
\begin{eqnarray*}
\sum_{b=m+1}^{\infty}{\la_{b}^{-1}}\leq{C}{m^{-\al}}
\end{eqnarray*}
for $m$ big enough, then for any $T>0$, with probability one the  $\La$-Fleming-Viot
process has a compact support at time $T$ and the Hausdorff dimension
for $\text{supp}(X(T)) $ is bounded from above by $2/\al$.
\end{corollary}
\begin{proof}
From the definitions of $\la_b$ and $\gamma_{b,m}$, we have
$\la_b\leq\gamma_{b,m}$ for any $b>m$. Thus
$\sum_{b=m+1}^{\infty}\la_{b}^{-1}\leq Cm^{-\al}$ implies
$\sum_{b=m+1}^{\infty}\gamma_{b,m}^{-1}\leq Cm^{-\al}$. Then the
conclusion is directly obtained from Theorem \ref{compact}.
\end{proof}

The lemma below on a lower bound for
the Hausdorff dimension can be found in Falconer \cite{Fal}.
\begin{lemma}\label{122271}
Let $A$ be any Borel subset of $\mathbb{R}^n$. If there is a
mass distribution $\mu$, supported by $A$ such that
\begin{equation*}
I_a\(A\)=\int_{\mathbb{R}^d}\int_{\mathbb{R}^d}\frac{1}{\left\|x-y\right\|^a}\mu\(dx\)\mu\(dy\)<\infty,
\end{equation*}
then $\dim\(A\)\geq a$.
\end{lemma}

By adapting the approach of Proposition 6.14 in Etheridge
\cite{Etheridge}, we could also find a lower bound on the Hausdorff
dimension for the support of $\La$-Fleming-Viot process at a fixed
time.
\begin{proposition}\label{122272}
Let $X$ be the $\La$-Fleming-Viot
process with underlying Brownian motion in $\bR^d$ for  $d\geq 2$. Then for any $T>0$, with probability one the
Hausdorff dimension of $\text{supp}(X(T))$ is at least $2$.
\end{proposition}
\begin{proof}
By Lemma \ref{122271}, we need to show that at fixed $T>0$,
\begin{equation*}
\EE\left[\int_{\mathbb{R}^d}\int_{\mathbb{R}^d}\frac{1}{\left\|x-y\right\|^a}X(T)\(dx\)X(T)\(dy\)\right]<\infty,
\end{equation*}
where $1<a<2$. Write $\left<\mu,f\right>$ for the integral of
function $f$ with respect to measure $\mu$. It is well-known that
moments of the $\La$-Fleming-Viot process can be expressed in terms
of a dual process involving $\Lambda$-coalescent and heat flow, see
Section 5.2 of  \cite{MJM} for such a dual process. For lack of
multiple collisions, expression for the second moment of the
$\La$-Fleming-Viot process is the same as that  for classical
Fleming-Viot process given in Proposition 2.27 of \cite{Etheridge}.
Then for any $\phi_1,\phi_2\in C_b(\mathbb{R}^d)$, we have
\begin{eqnarray*}
\begin{split}
\EE\left[\left<X(T),\phi_1\right>\left<X(T),\phi_2\right>\right]
=&e^{-rT}\left<X(0),P_T\phi_1\right>\left<X(0),P_T\phi_2\right>\\
&+\left<X(0), \int_0^Tre^{-rs}P_{T-s}\(P_s\phi_1 P_s\phi_2\)ds\right>,
\end{split}
\end{eqnarray*}
where $P_s$ is the heat flow and $r$ is the total coalescence rate
when the number of existing blocks is $2$, i.e., $r=\la_{2}$.

Following  arguments similar to Proposition 6.14 of \cite{Etheridge},
we can show that for any nonnegative function of the form $\psi\(x,y\)$,
\begin{equation*}
\begin{split}
&\EE\left[\int_{\mathbb{R}^d}\int_{\mathbb{R}^d}\psi\(x,y\)X(T)(dx)X(T)(dy)\right]\\
=&e^{-rT}\int\cdots\int
p\(T,z,w\)p\(T,z^{'},w^{'}\)\psi\(w,w^{'}\)dwdw^{'}X(0)(dz)X(0)(dz^{'})\\
&+r\int_0^T\int\cdots\int
e^{-rs}p\(T-s,z,w\)p\(s,w,y\)p\(s,w,y^{'}\)\psi\(y,y^{'}\)dydy^{'}dwX(0)(dz)ds,
\end{split}
\end{equation*}
where $p\(\cdot,\cdot,\cdot\)$ denotes the heat kernel.

Choose $\psi\(x,y\)=1/\left\|x-y\right\|^a$ for $1<a<2$. Following
the hint in proof of Proposition 6.14 in \cite{Etheridge} we can
show that both integrals on the right hand side of the above
equation are finite. Therefore, the Hausdorff dimension for the
support is at least $2$.
\end{proof}

\begin{corollary}\label{4.11}
Suppose that $d\geq 2$ and $\Lambda(\{0\})>0$, i.e., the
$\Lambda$-coalescent has a nontrivial Kingman component. Then at any
fixed time $T>0$, with probability one the $\Lambda$-Fleming-Viot
process has a compact support  of  Hausdorff dimension $2$.
\end{corollary}
\begin{proof}
Since $\Lambda(\{0\})>0$, the $\Lambda$-coalescent has a nontrivial
Kingman component. Then
\[\la_b\geq \frac{\Lambda(\{0\})b(b-1)}{2}\]
and
\[\sum_{b=m+1}^{\infty}\frac{1}{\la_b}\leq\sum_{b=m+1}^{\infty}\frac{2}{\Lambda(\{0\})b(b-1)}
=\frac{2}{\Lambda(\{0\})m}. \]
Applying Corollary \ref{120111} with $\al=1$, the
$\Lambda$-Fleming-Viot process has a compact support and
the Hausdorff dimension for the support is bounded from above by $2$ at any fixed time $T$. This together with Proposition \ref{122272} implies the desired result.
\end{proof}

\begin{remark}
Corollary \ref{4.11} complements the result on
Hausdorff dimension  for the classical Fleming-Viot
process in Dawson and Hochberg \cite{Dawson}.
\end{remark}
\subsection{Examples}\label{4.5}
\subsubsection{The $\La$-Fleming-Viot process with its coalescent having the $\(c,\epsilon,\gamma\)$-property}
\begin{lemma}\label{total rate}
For $n\geq 2$, there exists a positive constant $C(c,\gamma,\epsilon)$
such that the total coalescence rate of the $\La$-coalescent with the $(c,\ep,\ga)$-property
satisfies
$$\lambda_n\geq C(c,\gamma,\ep)n^{1+\gamma},$$ where
$$C(c,\gamma,\ep)=\frac{c\ep^{1-\gamma}}{2\(1-\gamma\)}\(\frac{1}{3\(2-\gamma\)}\)^{\gamma}e^{-\frac{\gamma^2}{2\(1-\gamma\)}}.$$
\end{lemma}
\begin{proof}
By the definition of $\la_n$, we have
\begin{eqnarray*}
\begin{split}
\la_n&=\sum_{k=2}^n{n\choose k}\la_{n,k}\geq{n\choose 2}\la_{n,2}\\
&\geq c{n\choose
2}\int_0^{\epsilon}x^{-\gamma}(1-x)^{n-2}dx\\
&= c{n\choose
2}\int_0^{1}(y\epsilon)^{-\gamma}(1-y\epsilon)^{n-2}\epsilon dy\\
&\geq c{n\choose 2}\epsilon^{1-\gamma}\int_0^{1}y^{-\gamma}(1-y)^{n-2}dy\\
&=\frac{c\ep^{1-\gamma}n}{2}\frac{(n-1)!\Gamma(1-\gamma)}{\Gamma(n-\gamma)}\\
&\equiv\frac{c\ep^{1-\gamma}n}{2(1-\gamma)}\times B,
\end{split}
\end{eqnarray*}
where
$$B=\frac{n-1}{n-1-\gamma}\times\frac{n-2}{n-2-\gamma}\times\cdots\times\frac{3}{3-\gamma}\times\frac{2}{2-\gamma}.$$
It follows from the inequality $\ln(1+x)\geq x-{x^2}/{2}$ for $0<x<1$ that
\begin{eqnarray*}
\begin{split}
\ln
B&=\sum_{l=2}^{n-1}\ln\(\frac{l}{l-\gamma}\)=\sum_{l=2}^{n-1}\ln\(1+\frac{\gamma}{l-\gamma}\)\\
&\geq\sum_{l=2}^{n-1}\frac{\gamma}{l-\gamma}-\frac{\gamma^2}{2}\sum_{l=2}^{n-1}\frac{1}{\(l-\gamma\)^2}\\
&\geq\int_2^n\frac{\gamma}{x-\gamma}dx-\frac{\gamma^2}{2}\int_1^{n-1}\frac{1}{\(x-\gamma\)^2}dx\\
&=\gamma\ln\frac{n-\gamma}{2-\gamma}-\frac{\gamma^2}{2}\(\frac{1}{1-\gamma}-\frac{1}{n-1-\gamma}\)\\
&\geq\gamma\ln\frac{n-\gamma}{2-\gamma}-\frac{\gamma^2}{2\(1-\gamma\)}.
\end{split}
\end{eqnarray*}
Consequently,
\begin{eqnarray*}
\begin{split}
\la_n&\geq\frac{c\ep^{1-\gamma}n}{2\(1-\gamma\)}\(\frac{n-\gamma}{2-\gamma}\)^{\gamma}e^{-\frac{\gamma^2}{2\(1-\gamma\)}}.
\end{split}
\end{eqnarray*}
Since $\gamma\in\(0, 1\)$, then $n-\gamma\geq n/3$ for any $n\geq
2$. Therefore,
\begin{eqnarray*}
\begin{split}
\la_n&\geq\frac{c\ep^{1-\gamma}n}{2\(1-\gamma\)}\(\frac{n}{3\(2-\gamma\)}\)^{\gamma}e^{-\frac{\gamma^2}{2\(1-\gamma\)}}
\equiv C\(c,\gamma, \ep\)n^{1+\gamma}.
\end{split}
\end{eqnarray*}
\end{proof}

\begin{proposition}\label{5.11}
Let $X$ be any $\La$-Fleming-Viot process with underlying Brownian motion in $\RR^d$ for $d\geq2$. If the corresponding $\La$-coalescent has the $(c,\ep,\gamma)$-property, then for any $T>0$, with probability one the random measure $X$
has a compact support at time $T$. Further,
\[2\leq\text{dim} \(\text{supp}\(X\(T\)\)\)\leq {2}/{\gamma}.\]
\end{proposition}
\begin{proof}
It follows from Lemma \ref{total rate} that
\begin{eqnarray*}
\begin{split}
\sum_{k=m+1}^{\infty}\la^{-1}_k
&\leq\sum_{k=m+1}^{\infty}\frac{1}{C\(c,\gamma, \ep\)k^{1+\gamma}}\\
&\leq\int_{m}^{\infty}\frac{1}{C\(c,\gamma, \ep\)x^{1+\gamma}}dx\\
&=\frac{1}{\gamma C\(c,\gamma,\ep\)m^{\gamma}}.
\end{split}
\end{eqnarray*}
Applying Corollary \ref{120111} and Proposition \ref{122272}, the
conclusion is immediately available.
\end{proof}

\subsubsection{The Beta$(2-\be,\be)$-Fleming-Viot process with underlying Brownian motion}
\begin{proposition}\label{Ex_beta}
Suppose that $d\geq2$. For any $T>0$, with probability one the
Beta$(2-\be,\be)$-Fleming-Viot process $X$ with underlying Brownian motion in $\RR^d$ has a compact support at time $T$
if and only if $\be\in\left(1,2\right)$. Further, for
$\be\in(1,2)$,
\[2\leq\text{dim} \(\text{supp}\(X\(T\)\)\)\leq 2/\(\be-1\).\]
\end{proposition}
\begin{proof}
For $\be\in(0,1]$, the corresponding Beta$(2-\be,\be)$-coalescent does not come
down from infinity.

For $\be\in(1,2)$, then $\be-1\in\(0,1\)$ and
given $\ep\in(0,1)$, for all $x\in[0,\ep]$, we have
\begin{eqnarray*}
\begin{split}
\La(dx)&=\frac{\Gamma(2)}{\Gamma(2-\be)\Gamma(\be)}x^{1-\be}\(1-x\)^{\be-1}dx\\
&\geq
\frac{\Gamma(2)\(1-\ep\)^{\be-1}}{\Gamma(2-\be)\Gamma(\be)}x^{1-\be}dx,
\end{split}
\end{eqnarray*}
which implies the Beta$\(2-\be,\be\)$-coalescent has the $(c,\ep,\be-1)$-property.

By Proposition 7.2 of \cite{Bla09} and Proposition \ref{5.11}, the
Beta$(2-\be,\be)$-Fleming-Viot process has a compact support if and
only if $\be\in\left(1,2\right)$ and the Hausdorff dimension for its
support is between $2$ and $2/\(\be-1\)$.
\end{proof}
\begin{remark}
Intuitively, since the Beta-coalescent comes down from infinity at a
speed slower than Kingman's coalescent, the particles in the
lookdown representation are less correlated. So we expect a higher
Hausdorff dimension for the support of  Beta-Fleming-Viot process
with underlying Brownian motion.
\end{remark}
\begin{remark}
By Proposition \ref{Ex_beta} the coming down from infinity property
is equivalent to the compact support property  for
Beta$(2-\be,\be)$-Fleming-Viot processes, which suggests that the
assumption (\ref{VI}) is rather mild.
\end{remark}

\noindent {\bf Acknowledgement} The authors are thankful to
anonymous referees for very helpful suggestions and detailed
comments.

\end{document}